\documentclass[10pt]{article}

\usepackage{amssymb}
\usepackage{amsfonts}
\usepackage{amsmath}
\usepackage{amsthm}
\usepackage{hyperref}
\usepackage{xcolor}
\usepackage{yfonts}
\usepackage{graphicx,epstopdf}
\usepackage{tikz}

\usepackage{tikz-3dplot}

\usepackage{lscape}

\usepackage{xcolor}
\usepackage{yfonts}
\usepackage{hyperref}
\usepackage{float}
\usepackage{enumitem}

\usetikzlibrary{shapes.geometric, calc}
\usetikzlibrary{decorations.markings,shapes.arrows, shapes.symbols}
\usetikzlibrary{decorations.pathreplacing}

\newtheorem{theorem}{Theorem}[section]
\newtheorem{proposition}[theorem]{Proposition}
\newtheorem{lemma}[theorem]{Lemma}
\newtheorem{corollary}[theorem]{Corollary}

\theoremstyle{definition}
\newtheorem{definition}[theorem]{Definition}
\newtheorem{example}[theorem]{Example}
\newtheorem{remark}[theorem]{Remark}

\renewenvironment{proof}[1][Proof]{\textbf{#1.} }{\ \rule{0.5em}{0.5em}}

\newcommand{\ceil}[1]{\left \lceil{#1}\right \rceil}
\newcommand{\PG}{P}

\begin{document}

\author{Yuri Muranov \\
University  of  Warmia and \\
Mazury in Olsztyn,\\
Poland \\
\and 
Anna Muranova \\
University  of  Warmia and \\
Mazury in Olsztyn,\\
Poland \\
}
\title{Burnings of trees and their homologies}
\date{\today}
\maketitle

\begin{abstract}The problem of graph burning was firstly introduced as a  model for different processes of social and network interactions.  Recently, the authors of the  present paper developed   methods of algebraic topology for investigation of this problem.  This approach is  based on the new definition of burning process which excludes the possibility to choose at any moment vertex for burning from the set of vertices which are already burned at this moment.

In this paper we continue to study such burning process using algebraic topology methods.  We prove the result about relations between burnings of a graph and burnings of its  spanning  trees that is similar to the classical case.  Afterwards,  we describe properties of trees burnings. In particular,   we prove that a burning of a tree defines a structure of a digraph on the tree and investigate this structure. We introduce and study  a  strong burning  configuration space of a graph and new strong  burning homology  which are similar to burning homology defined in  our previous paper, but arise from burning homomorphism.
\end{abstract}

\bigskip



\section{Introduction}\label{S1}
\setcounter{equation}{0}

The problem of graph burning was firstly introduced in \cite{Burning_2014}. This model describes different processes of social and network interactions, see also   \cite{Bonato},   \cite{Bonato_0}.  The significance of trees in the burning theory follows from the 
results \cite[Theorem 2.4, Corollary 2.5]{Bonato_0}  that describe relations between   burning  of a connected graph $G$ and  burnings of some  of its  spanning trees.  In particular, as follows from this result, any burning of $G$ gives a burning of a spanning tree, although a spanning tree is not uniquely determined by the burning in general.

The definition of a burning process  in \cite{Burning_2014, Bonato_0}  allows that in  time $t$ a  vertex which is burned in time $t+1$ can be also chosen to burn  in time $t+1$ (see \cite{MMBurning2025} for details). In  \cite{MMBurning2025} we introduced several filtrations of a graph, closely related to the burning process, and defined 
the burning process which excludes the possibility to choose at the moment $t+1$ vertex for burning from the set of vertices which are already burned at this moment.

	 In this paper we continue to study burning process in the sense  \cite{MMBurning2025}. We start with preliminaries in Section \ref{S2}. In Section \ref{S3} we prove the result about relations between burnings of a graph and burnings of its  spanning  trees that is similar to \cite[Theorem 2.4, Corollary  2.5]{Bonato_0}.  The main result in this direction is given by Theorem \ref{thm::spanning tree}.  But there is an essential difference between our proof and the proof of  \cite[Theorem 2.4]{Bonato_0} as we note in Remark \ref{Difference}. Therefore, Theorem~\ref{thm::spanning tree} reduces investigation of   graph burnings to  burnings of trees. In Section \ref{S4} we define burning homomorphisms and discuss their existence on trees. In Section~\ref{S5} we describe possible end time of trees burnings and properties of burning sequences in the case of burning homomorphisms. In particular, we obtain  sharp lower and upper bounds of the burning homomorphism's end time for path graphs $P_n,n\in \Bbb N$, Theorem \ref{thm::bounds}. In Section \ref{S6}  we prove that a burning of a tree defines a structure of a digraph on the tree and investigate properties of this structure. We finish with Section \ref{S7}, where we introduce and study  a  strong burning  configuration space of a graph and a strong homology  which are similar to burning homology defined in   \cite{MMBurning2025}.

\section{Preliminaries}\label{S2}
\setcounter{equation}{0}
\subsection{Basic notions of the graph theory}
In this section, we recall basic notions of  the graph theory, see e. g. \cite{Gary, MiHomotopy, Mi3, MiHH}.

\begin{definition}\label{d2.1} 
 \emph{A graph} $G=(V_G,E_G)$ is a non-empty finite set $V_G$ of  \emph{vertices} together with a set $E_G$ of non-ordered pairs $\{v,w\}\in E_G$ of distinct vertices $v,w\in V_G$  called \emph{edges}. Two different  vertices are \emph{adjacent} if they are connected by an edge. A vertex is \emph{incident} with an edge 
 if the vertex is one of the endpoints of that edge. All vertices which are incident to vertex $v$ are called \emph{neighbors} of $v$. 

\end{definition}

Sometimes we will omit 
the subscript $G$ from the set  $V_G$ of vertices and from the set $E_G$ of edges if the graph is clear from the context.

\begin{definition}\label{d2.2} 
We say that a  graph  $H$ is a  \emph{subgraph} of a graph $G$  and  we write $H\subset G$ if  $V_H\subset V_G$ and $E_H\subset E_G$. A subgraph 
$H\subset G$ is an \emph{induced subgraph} if for any edge $\{v,w\}\in E_G$,  such that $v, w\in V_H$ we have $\{v,w\}\in E_H$. In this case we write 
$H\sqsubset G$.  For a subgraph $H\subset G$ we denote by $\widehat{H}\sqsubset G$ the induced subgraph which has the same set of vertices as $H$. 
\end{definition}

\begin{definition} \label{d2.3} A \emph{graph map $f\colon G\to H$}  from a graph
$G$\emph{\ }to a graph $H$ is a map of vertices $f|_{V_G}\colon
V_{G}\rightarrow V_{H}$ such that for any edge   $\{v,w\}\in E_G$,  we have   
$\{f(v),f(w)\}\in E_H$  or $f(v) =f(w)\in V_H$. The map
$f$ is called a \emph{homomorphism}  if $\{f(v),f(w)\}
\in E_H$ for any $\{v,w\}\in E_G$. For a graph $G$,  we  denote by $\operatorname{Id}_G\colon G\to G$ the \emph{identity} map (homomorphism)  that is the identity map on the set of vertices and  on the set  of edges. 
\end{definition}

\begin{definition}\label{d2.4} 
(i) For $n\geq 1$,  a \emph{path in a graph} $G=(V,E)$ is an alternating sequence $v_0, a_1, v_1, a_2, \dots , v_{n-1}, a_n, v_n$ of vertices 
$v_i\in V$ and edges $a_i\in E$ such that $a_i=\{v_{i-1}, v_i\}$ for $i=1, \dots, n$ and $v_i\ne v_j$ for all distinct $i,j=0, \dots, n$.  
For $n=0$,  a path is given by a vertex $v_0\in V$. The integer $n$ is the \emph{length} of the path.  The vertex $v_0$ is  the \emph{origin}  and the vertex $v_n$ is the \emph{end} of the path. 

(ii) For $n\geq 3$,  a \emph{circuit} in a graph $G=(V,E)$ is an alternating sequence $v_0, a_1, v_1, a_2, \dots , v_{n-1}, a_n, v_0=v_n$ of vertices 
$v_i\in V$ and edges $a_i\in E$ such that $a_i=\{v_{i-1}, v_i\}$ for $i=1, \dots, n$ and $v_i\ne v_j$ for all distinct $i,j=1, \dots, n$.  For $n<3$,  a circuit is not defined.
\end{definition} 

\begin{definition}\label{d2.5} 
(i) A  graph  $G=(V,E)$ is \emph{connected}
if for any two distinct vertices $v, w\in V$ there is a path  for which $v_0=v$ is the  origin and $v_n=w$ is the end of the path.

(ii) A  connected graph $G$ is a \emph{tree} if it is does not  contain  circuits.  

(iii)
A \emph{path graph} $P_n$, where $n$ is a positive integer, is a graph with 
the set of vertices $\{1,\dots, n\}$ and with the set of edges
$\{i,i+1\}$ for $i=1, \dots, n-1$.

\end{definition}

\begin{definition}\label{d::subtrees} 
Let  $G=(V,E)$ be a graph.

(i)  A subgraph $T$ of $G$ is its  {\em subtree} if it is a tree.

(ii) A subtree $T=(V_T, E_T)$  of $G$ is its  {\em spanning tree} if $V_T=V$.

(iii) A tree $Q$ is {\em an ambient tree} of $G$, if $G$ is a subtree of $Q$.

\end{definition} 

The following lemma is well known and follows almost immediately from definitions above.
\begin{lemma}\label{l:onePath}
A graph $T=(V_T, E_T)$ is a tree if and only if for any two distinct vertices $v,w\in V_T$ there exists exactly one path $v = v_0, a_1, v_1, a_2, \dots , v_n =~w$.
\end{lemma}

\begin{definition}\label{d2.6} For a graph $G=(V,E)$, the \emph{distance} $d_G(v,w)$ between two vertices $v,w\in V$ is the number of edges in a shortest path connecting these vertices.  If there is no path connecting  vertices $v,w\in V$ then we set $d_G(v,w)=\infty$.  
\end{definition}

\begin{definition}\label{d2.7} Let $G=(V_G,E_G)$ be a graph and $k\in \mathbb N$. 

(i) For an integer $n\geq 0$, the \emph{n-th closed neighborhood} of a vertex $v\in V_G$ is an induced subgraph with the set of vertices $\{w\in V_G| \, d_G(v,w)\leq n\}$. We denote this neighborhood by $N_n(v)$. 

(ii) The \emph{$n$-th closed neighborhood} of an induced subgraph $H\sqsubset G$  is the set of vertices 
$
\{w\in V_G| \, \exists \, v\in V_H \text{ such that }d_G(v,w)\leq n \}.
$
We denote this neighborhood by $N_n(H)$. 

\end{definition}

\begin{definition}\label{d2.8}  Let $H_i=(V_i, E_i)\, (i=1,2)$ be two subgraphs of a graph $G$. The \emph{induced union} $H_1\Cup H_2$ is the induced subgraph of $G$ generated by the set of vertices  $V_1\cup V_2$. It  equals  to the induced subgraph $\widehat{H_1\cup H_2}$. 
\end{definition} 

\begin{definition}
Let $G = (V_G, E_G)$ and $H = (V_H , E_H)$ be two graphs.
The disjoint union $G + H = (V, E)$ is a graph which has as its vertices the
set $V = V_G \sqcup V_H$ and as its edges the set $E = E_G \sqcup E_H$.
\end{definition}

\subsection{Graph burning and corresponding graph filtrations}
\setcounter{equation}{0}

In this section we remind  the definition of  \emph{graph burning} from \cite{MMBurning2025} and related concepts.

 Let $G=(V_G,E_G)$ be a graph and 
$S_G=(v_1, \dots, v_n)$ be a non-empty ordered  set of vertices. For $1\leq j\leq n$,
consider a set of induced  subgraphs $\mathcal N_j\sqsubset G$ defined by 
\begin{equation}\label{3.1}
\mathcal N_j=N_{j-1} (v_1)\Cup N_{j-2}(v_2)\Cup \dots \Cup N_1(v_{j-1})\Cup N_0(v_j). 
\end{equation}
We note that 
\begin{equation*}
\begin{matrix}
\mathcal N_1=\{v_1\}, \ \mathcal N_2=N_1(v_1)\Cup \{v_2\},  \ \text{and} \ \ \mathcal N_3=N_2(v_1)\Cup N_1(v_2)\Cup \{v_3\}. 
\end{matrix}
\end{equation*}

For $2\leq j\leq n+1$, let  $U_j$ be  the induced subgraph 
of $G$ given by   
\begin{equation}\label{3.2}
U_j=N_{j-1} (v_1)\Cup N_{j-2}(v_2)\Cup \dots \Cup N_1(v_{j-1}).
\end{equation}
We note that  for $2\leq j\leq n$, the set of vertices $V_{U_j}$ of the graph $U_j$ is obtained from the set of vertices $V_{\mathcal N_j}$ of the graph $\mathcal N_j$  by deleting the vertex $v_j$.

It follows immediately from (\ref{3.1}) and (\ref{3.2}) that we have the following  commutative diagram of inclusions of induced subgraphs 
\begin{equation}\label{3.3}
\begin{matrix}
\mathcal N_1 &\sqsubset & \mathcal N_2& \sqsubset &\mathcal N_3&\sqsubset \dots \sqsubset  &\mathcal N_n&\sqsubset&G&\\
 && \sqcup&  &\sqcup &\dots & \sqcup&&\sqcup& \\
                  &  &U_2&\sqsubset& U_3&\sqsubset\dots \sqsubset  &U_n&\sqsubset&U_{n+1}&\sqsubset  G\\
&& \sqcup&  &\sqcup &\dots & \sqcup&&\sqcup&\ \ \  || \\
                  &  &\mathcal N_1&\sqsubset& \mathcal N_2&\sqsubset\dots \sqsubset  &\mathcal N_{n-1}&\sqsubset&\mathcal N_{n}&\sqsubset  G\\
\end{matrix}
\end{equation}
in which the rows give two  filtrations of the graph  $G$ by induced subgraphs.

\begin{definition}\label{def::burningSeq} \rm Let $G=(V,E)$ be a graph and 
$S_G=(v_1, \dots, v_n)$ be a non-empty ordered  set of vertices. 
 The sequence $S_G$ is called a \emph{burning sequence}  if  in  diagram (\ref{3.3}) $v_j \notin U_j$ for $2\leq j\leq n$ and  $U_{n+1}=G$. 
 The elements of $S_G$ are called \emph{burning sources}. The vertex $v_i$ is called the \emph{source in the time $i$}.  For every burning sequence $S_G$ we denote $\widehat S_G = \{v_1, \dots, v_n\}\subset ~V$ the set of corresponding burning sources.
\end{definition}

Let $\mathbb N$ be the set of positive integers. A burning sequence $S_G$ define the commutative diagram  (\ref{3.3}).  Setting 
$
\mathcal N_{n+1}\colon =U_{n+1}=G
$
we obtain a filtration of $G$ 
\begin{equation}\label{3.4}
\mathcal N_1 \sqsubset  \mathcal N_2\sqsubset \dots \sqsubset  \mathcal N_n\sqsubset \mathcal N_{n+1}=G
\end{equation}
which corresponds to  the burning sequence  $S_G$. We define a function 
 $\lambda_G\colon V_G\to \mathbb N$, where  
\begin{equation}\label{3.5} 
 \lambda_G(v)=\min\{i \,| \, v\in \mathcal N_i\}
\end{equation}
It follows immediately from   filtration (\ref{3.4})   that   the function $\lambda_G$ is well defined.

We call a pair $\mathbf B(G)\colon =(\lambda_G, S_G)$  a \emph{burning (process)} for the graph $G$. Further,   for a vertex $v\in V_G$ the value $\lambda_G(v)$ we call  the \emph{burning time of the vertex} $v$.

\begin{theorem} [Lemma 3.2 and Theorem 3.5 in \cite{MMBurning2025}] \label{t3.5}  
 Let $G=(V,E)$ be a graph and $\mathbf B(G) =(\lambda_G, S_G)$ be its burning. Then the function $\lambda_G=\lambda(\mathbf  B(G))$   is a surjective map from  the set $V$ of vertices to the set  $\{1,2,\dots, {\mathbf T}\}$,  
where  ${\mathbf T}=n $ or ${\mathbf T}=n+1$. Moreover, $\lambda_G$ defines an unique graph   map $G\to \PG_{\mathbf T}$  which coincides 
with  $\lambda_G$ on the set of vertices $V$ of the graph $G$. We continue to denote this graph map  by $\lambda_G=\lambda(\mathbf  B(G))$. 
\end{theorem}
 We call the integer $\mathbf T$   the \emph{end time of the burning}.   We write $\mathbf B^{\mathbf T}(G)$   if the process ends in the  time ${\mathbf T}$. We denote  by $\lambda(\mathbf  B^{\mathbf T}(G))$ the function $\lambda_G$ which corresponds to the  process $\mathbf B^{\mathbf T}(G)$.

\begin{lemma}[Lemma 3.4 in \cite{MMBurning2025}]
\label{l3.4}
Let $G=(V,E)$ be a graph with a burning  $\mathbf  B(G)=(\lambda, S_G)$.    Let  $v,w\in V$ be  neighbours, that  is $\{v,w\}\in E$.  Then for any two  neighbors vertices $v,w\in V$  we have $|\lambda(v)-\lambda(w)|\in\{0,1\}$.
\end{lemma}

\begin{definition}\label{d3.6} Let $G$ be a graph and    $\mathbf  B^{\mathbf T}(G)$   be a burning process. The graph map (homomorphism) $\lambda=\lambda(\mathbf  B^{\mathbf T}(G))$ is called the \emph{burning map (homomorphism)}.   
\end{definition}

We also will need the following immediate corollary of Lemma \ref{l3.4} and Definition \ref{d2.3}.

\begin{corollary}\label{cor:l3.4}
Let $G=(V,E)$ be a graph and $\mathbf  B^{\mathbf T}(G)$ be its burning.  Then the map $\lambda=\lambda(\mathbf  B^{\mathbf T}(G))$ is a burning homomorphism if and only if for any two adjacent vertices $v,w\in V$  we have $|\lambda(v)-\lambda(w)|=1$.
\end{corollary}

Finally, we remind that there are functorial relations between burnings of various graphs. Let $(\lambda_G, S_G)$ and  $(\lambda_H, S_H)$  be two  burnings  of  graphs $G$  and $H$, respectively,  with the source sequences  $S_G=(v_1, \dots, v_k)$  and $S_H=~(w_1, \dots, w_m)$  where $m\geq k\geq 1$. 
 For $2\leq j\leq  k+1$, let $U_j$   be the induced subgraph of $G$ defined in (\ref{3.2}).  For $2\leq j\leq  k+1$,  let 
$W_j\sqsubset H $  be the induced subgraph of $H$ defined  by
\begin{equation}\label{3.11}
W_j=N_{j-1} (w_1)\Cup N_{j-2}(w_2)\Cup \dots \Cup N_1(w_{j-1}).
\end{equation}

\begin{definition}\label{d3.15} \rm  Let  $G$, $H$ be graphs and
$\mathbf B(G)$,  $\mathbf B(H)$ be two  burnings  as above.
  A graph map $f\colon G\to H$ is called a \emph{morphism of the burnings}, and we write $f\colon \mathbf B(G)\to \mathbf B(H)$,   if  $k\leq m$,    $f(v_i)=w_i$ for  $1\leq i\leq k$,  and  there is a  commutative diagram 
of graph maps
\begin{equation}\label{3.12}
\begin{matrix}
G &\overset{f}\longrightarrow & H\\
\lambda_G\downarrow\ \ \  \ &&\ \ \ \downarrow\lambda_H\\
P_{T^G} &\overset {\tau}\longrightarrow & P_{T^H}, \\
\end{matrix}
\end{equation}
where $T^G$ and $T^H$ are the burning times for the burnings  of  the graphs $G$ and $H$,  respectively. Hence, in particular,  $f(U_j)\subset W_j$ for $2\leq j\leq k+1$.
\end{definition}

\section{Burning of a spanning tree}\label{S3}
\setcounter{equation}{0}

In this section we describe relations between burnings  of a graph and burnings  of its  spanning  trees for the burnings defined in previous section. The main result of this section is Theorem \ref{thm::spanning tree} , which states that any burning of a graph $G$ defines a burning of its  spanning tree, which is in general not uniquely determined. Recall that a  \emph{spanning tree}    of  a connected graph  $G$ is a subgraph tree which contains  all vertices of  $G$. 

The burning process in Definition \ref{def::burningSeq} differs from the burning process described in  \cite[Introduction]{Bonato_0}.  In particular, as follows from our definition,   the distance between sources $v_i$ and $v_j$ satisfies condition $d(v_i, v_j)\geq |i-j|+1$ in contrast to the situation  in  \cite{Bonato_0} where  $d(v_i, v_j)\geq |i-j|$.  The proof of Theorem \ref{thm::spanning tree} below  is  similar  to that  given in   \cite{Bonato_0} but there is an essential difference  as we note in Remark \ref{Difference}. At first, we state and prove a  technical lemma. 

\begin{lemma}\label{l6.1}  Let $G=(V,E)$ be a connected graph and $S_G$ be a burning sequence $(v_1, \dots, v_n)$ of a burning $(\lambda, S_G)$ of $G$. Then 
for any vertex $w\in V$ we have 
\begin{equation}\label{6.1}
d(v_k, w)\geq  |\lambda(v_k) -\lambda( w)|.
\end{equation}
\end{lemma} 
\begin{proof}  Let $p$ be a path from the vertex $v_k$ to the vertex $w$ of length $d=d(v_k, w)$ which is finite, since the graph is connected.    By Theorem \ref{t3.5}, the function $\lambda$ is  a graph map 
$\lambda\colon G\to P_{\mathbf T}$ where $P_{\mathbf T}$ is a path graph.  The restriction of 
$\lambda$ to the path $p$ gives the path map $\lambda_P\colon P\to P_{\mathbf T}$ of the subgraph $P\subset G$ which is generated by $p$. The graph $P$ is a path graph  of the length $d$ and its image is the path subgraph of $P_{\mathbf T}$. Hence 
the distance between images $\lambda(v_k)$ and $\lambda(w)$  of the vertices 
$v_k, w\in V_P$ is less or equal to the  distance $d(v_k, w)$. That is 
$d(v_k, w)\geq  |\lambda(v_k) -\lambda( w)|.$
\end{proof}

\begin{theorem}\label{thm::spanning tree} Let $G=(V,E)$ be a connected graph and $S_G$ be a burning sequence $(v_1, \dots, v_n)$ of a burning $(\lambda, S_G)$ of $G$. Then there is a spanning tree $T\subset  G$  such that $S_G$ is a burning sequence for some burning of $T$. 
\end{theorem} 
\begin{proof} Let $\widehat S_G$ be the set  elements of $S_G$.  For any vertex $w\in G$  we define a set $V_w\subset \widehat S_G$  by 
\begin{equation}\label{6.2}
V_w=\{v_k\in S_G\,| \, \lambda(w)=\lambda(v_k)+d(v_k, w)\}.
\end{equation}
As follows from (\ref{6.2}), $v_j\in V_{v_j}$ since $\lambda(v_j)=j$ and $d(v_j, v_j)=0$.  
Let $2\leq j\leq n$. By Definition \ref{def::burningSeq},  the vertex  $v_j$  does not belong to $U_j$ defined in (\ref{3.2}). Hence, for $1\leq k\leq j-1$, 
$v_j\notin N_{j-k}(v_{k})$ by (\ref{3.2}). That is,  for such values of $k$, 
 $d(v_j, v_{k})>j-k$ and
 $$
\lambda(v_{k})+d(v_j, v_{k})=  k+d(v_j, v_{k}) > k+j-k =j=\lambda(v_j).
$$
Hence,  $v_k\notin  V_{v_j}$ for $1\leq k\leq j-1$. For $j<k\leq n$,   
 $d(v_j, v_{k})>1$ and 
$$
\lambda(v_{k})+d(v_j, v_{k})=  k+d(v_j, v_{k}) > j=\lambda(v_j). 
$$
Hence,  $v_k\notin  V_{v_j}$ for $k>j$. Thus we obtain that 
\begin{equation}\label{6.3}
V_{v_k}=\{v_k\}    \ \ \text{for every burning source} \ \  v_k\in \widehat S_G.
\end{equation}

We check now, that the set $V_w$ is non-empty for any vertex $w\in V$. Fix a vertex $w\in V$.  We  have proved already  this statement for $w\in \widehat S_G$. Fix a vertex $w\in V$ such that $w\notin \widehat S_G$ and let $\lambda(w)=j$. 
By Definition \ref{def::burningSeq}  and relation (\ref{3.2}), we have $U_{n+1}=G$. Let $U_1=\{v_1\}=N_0(v_1)$. 
By (\ref{3.2}) and \ref{3.1}),  there is a minimal value  $j\, (2\leq j\leq n+1)$ such that  
\begin{equation}\label{6.4}
w\in U_j=N_{j-1} (v_1)\Cup N_{j-2}(v_2)\Cup \dots \Cup N_2(v_{j-2})\Cup N_1(v_{j-1}) 
\end{equation}
and 
\begin{equation}\label{6.5}
w\notin U_{j-1}=\begin{cases}  N_0(v_1) & \text{for} \ \ j=2, \\
N_{j-2} (v_1)\Cup N_{j-3}(v_2)\Cup \dots \Cup N_1(v_{j-2}) &  \text{for} \  j>2.\\
\end{cases}
\end{equation}
Now, by (\ref{6.4}) and (\ref{6.5}), there is at least one number    $k \, (1\leq k\leq j-1)$ such that $w\in N_{j-k}(v_k)$ and $w\notin N_{j-k-1}(v_k)$. 
Hence, by Definitions \ref{d2.6} and \ref{d2.7},  we obtain that 
$d(w, v_k)=j-k$. Hence $j=\lambda(w)= \lambda(v_k)+d(w, v_k)$ and $v_k\in V_w$. 

For any vertex $w\in V,$ let $\mu (w)$ be a minimal value of $k$ such that $v_k\in V_w$. We obtain a function 
\begin{equation}\label{6.6}
\mu\colon V\to \widehat S_G
\end{equation}
which is well defined  since the map $\lambda$ is well defined and which is an epimorphism by (\ref{6.3}). For every vertex $v_k\in \widehat S_G$ define 
a subset of vertices $M_k\subset V$ by setting $M_k=\mu^{-1}(v_k)$. We note, that $v_k\in M_k$. By the definition, $M_i\cap M_j=\emptyset $ for $i\ne j$. Now, for any $1\leq k \leq n$,  we  construct a  subgraph  tree $T_k\subset G$ with the set of vertices 
$V_{T_k}=M_k$. 

\textbf{Claim}. Let $w\in M_k\, (1\leq k\leq n)$ and, hence,  $\lambda(w)=k+d(v_k, w)$.  Let $p$ be  a path of the length $d=d(v_k,w)$ with the pairwise distinct edges 
 from  $v_k$ to $w$ given by 
\begin{equation}\label{6.7}
p=(v_k=w_0, a_1, w_1, \dots , a_d, w_d=w) \ \  \text{where} \ \ a_i\in E, w_i\in V. 
\end{equation} 
  Then $w_i\in M_k$ for 
$0\leq i\leq d$. 

\begin{proof}  
Recall that we already know that $w_0, w_d\in M_k$. By Theorem \ref{t3.5}, the function $\lambda$ is  a graph map 
$\lambda\colon G\to P_{\mathbf T}$ where $P_{\mathbf T}$ the path graph.  The restriction of 
$\lambda$ to the path $p$ gives the path map $\lambda_P\colon P\to P_{\mathbf T}$ of the subgraph $P\subset G$ which is generated by the path $p$. The graph $P$ is a path graph  of the length $d$ and its image is the path subgraph of $P_{\mathbf T}$. We have $\lambda(w_0)=\lambda(v_k)=k$ and $\lambda(w_d) =\lambda(w) = k+d$ since $w\in M_k$. Hence the image under map $\lambda_P$ of the path of length $d$ is the path of the length $d$ and, therefore, $\lambda_P$ is the graph monomorphism.  Hence, for the vertices $w_i$ fitting in (\ref{6.7}) we have 
\begin{equation}\label{6.8} 
\lambda(w_i)=k+i \ \text{for} \ \ 0\leq i\leq d \ \  \text{and} \  v_k\in V_{w_i}.
\end{equation} 

Let us suppose that there is a vertex $w_j \, (0<j< d)$  of the path $p$ such  that $w_j\in M_q$ for $q< k$.  Then 
\begin{equation}\label{6.9}
\lambda(w_j)= k+j=\lambda(v_q)+d(v_q, w_j)= q+d(v_q, w_j). 
\end{equation}
Consider a path 
$$
p^{\prime}=(v_q=s_0, b_1, s_1, \dots, b_{d^{\prime}}, s_{d^{\prime}}=w_j) \ \text{where} \ s_{i}\in V, b_i\in E, d^{\prime}=d(v_q, w_j)
$$ 
from the vertex $v_q$ to the vertex $w_j$ of the length 
$d(v_q, w_j)$.  Similarly to case of the path $p$, we conclude that the restriction $\lambda_{p^{\prime}}\colon P^{\prime}\to  P_{\mathbf T}$ is a monomorphism where 
$P^{\prime}$ is the graph generated by the path $p^{\prime}$. Comparing the the restrictions of the map $\lambda$ to $P$ and $P^{\prime}$  and using inequalities $\lambda(v_q)< \lambda(w_j)< \lambda(w)$  we obtain that restriction of $\lambda$ to the path 
$$
(v_q=s_0, b_1, s_1, \dots, b_{d^{\prime}}, s_{d^{\prime}}=w_j, a_{j+1}, \dots , a_d, w_d=w)
$$
is a monomorphism. Hence 
$$
d(v_q, w)=d(v_q, w_j)+d(w_j, w)
$$
 and furher  using  (\ref{6.9}) and the relation $d(w_j, w)=d-j$ we obtain 
$$
\lambda(w)= k+d=(k+j)+(d-j)= \lambda(w_j)+d(w_j,w)
$$
$$
=\lambda(v_q)+d(v_q, w_j)+d(w_j,w)=\lambda(v_q)+d(v_q, w). 
$$
 Thus we obtain a contradiction with the minimality $k$  for  $w\in M_k$. 
\end{proof}

\bigskip

Now we can come back to the construction of a tree $T_k$ with the set of vertices $M_k$. For any $w\in M_k$ we denote $\chi(w):=d(v_k, w)$ and fix a path $p_w=(v_k, a_1^w, \dots , a^{w}_{\chi(w)} , w)$  in $G$ from $v_k\in M_k$ to $w\in M_k$. Let
$$
m = \max_{w\in M_k}\chi(w).
$$
We define the sequence of rooted trees 
\begin{equation}\label{6,10}
\{v_k\}=T_k^0\subset T_k^1\subset \dots \subset T_k^{m}=T_k 
\end{equation}
as follows.  The tree $T_k^1=(M_1, E_1)$ has the set of edges 
$$
E_1=\{\{v_k, w\}\, | \, \chi(w)=1\}.
$$
Further,  we define $T_k^j=(M_j, E_j)$ where $E_j$ is defined inductively as
$$
E_j=E_{j-1}\cup\{ a_{\chi(w)}^w \, | a_{\chi(w)}^w\ \text{is the last edge of}\  p_w\}. 
$$
It follows from the claim above, that  the set $E_j$ is well defined. Moreover, by induction from the claim above follows that $T_k^j$ is connected.  Moreover, $T_k^j$  is  a tree due to Lemma \ref{l:onePath}, since by induction  there is only one path between two vertices in $T_k^j$.

Now we  define the spanning tree of the graph $G$  using the trees 
$T_k\, (1\leq k\leq n)$. For every two distinct trees  $T_k, T_l \, (1\leq k,l\leq n)$,  
we denote by $E_{k,l}$ the set of vertices that has a form $\{x,y\}$ where 
$x\in M_k, y\in M_l$. Note that this set can be empty for some pairs $(k,l)$. 
In every set $E_{k,l}$ fix an edge $a_{k,l}$. Now define a quotient graph 
$\mathbf G=G/\sim\, $ with the set of vertices $\{t_i\, |\,  1\leq i\leq n\}$ and the set of edges given by such pairs of vertices $\{t_k, t_l\}$ for which the edge $a_{k,l}$ 
exists. The graph $\mathbf G$ is connected  since the graph $G$ is connected. Consider a spanning tree $\mathbf Q$ of the graph $\mathbf G$.   Now we  define a spanning tree $T$ of the graph $G$ by setting 
$$
T=(V_T, E_T) \ \ \text{where} \  V_T=V, \  
E_T=\bigcup_{1\leq k\leq n} E_{T_k}\bigcup\{a_{k,l} \, | \, \{t_k, t_l\}\in E_{\mathbf Q}\}. 
$$
It follows immediately from our construction and Definition \ref{def::burningSeq} that the burning sequence $S_G$ is a burning sequence of the tree $T$.
\end{proof}

\begin{remark}\label{Difference}
 The key difference between Theorem \ref{thm::spanning tree}  above and  the theorem, stated by Bonato and coauthors \cite[Th. 2.4]{Bonato_0}  is that  in our settings all the trees $T_j$ are not empty. It is possible due to the fact that we do not allow a source $v_j$ to be the vertex, which anyway would be burned at the step $j$ from some other source. Let us illustrate the difference with an example. Consider a path graph $P_4$ with the vertices $1,2,3,4$ and edges $\{i, {i+1}\}$.
By Bonato et al. the burning sequence $(1,2,3,4)$ is allowed, which leads to $T_2, T_3, T_4$ being empty. On the contrast, in our setting the only allowed burning sequences up to isomorphism are $(1,3)$, $(2,4)$ and $(1, 4)$, for each of which both trees $T_1$ and $T_2$ are not empty.
\end{remark}
We can omit the last step, i.e. the construction of spanning tree of $G$, at the end of the proof of Theorem \ref{thm::spanning tree}. Then we obtain the following result.

\begin{corollary}\label{cor:cupTi}  Let $G=(V,E)$ be a  graph and $S_G$ be a burning sequence $(v_1, \dots, v_n)$ of a burning $(\lambda, S_G)$ of $G$. Then there are  trees  $T_i\subset  G, i=1\dots, n$ such that $V_{\sqcup T_i}=V$ and $S_G$ is the burning sequence for a burning of the graph $H= T_1+T_2+\dots+T_n$. Moreover, the corresponding burning map  $\lambda=\lambda_ {H}$ is a homomorphism.
\end{corollary} 
\begin{proof}
The fact that $S_G$ is the burning sequence of a burning of 
$H=(V_H, E_H)$ is clear. Assume there exists two adjacent vertices $v,w\in \lambda_H$ such that $\lambda(v)=\lambda(w)$. Be definition of $H$, $v, w$ are in the same $ T_i$ and, hence, $v,w\in M_i$. Moreover, $d(v_i,v)+i=\lambda(v)=\lambda(w)=d(v_i,w)+i$, i.e. $d(v_i,w)=d(v_i,v)$ and $v\sim w$, which is impossible in a tree due to Lemma~\ref{l:onePath}.
\end{proof}

\section{Burning homomorphisms and trees} \label{S4}

In this section we study burning homomorphisms $\lambda \colon G\to P_T$ of graphs and their properties for trees.

\begin{lemma}\label{lem::ge3}
 Let $G=(V,E)$ be a  graph and $S_G$ be a burning sequence $(v_1, \dots, v_n)$ of a burning homomorphism $(\lambda, S_G)$ of $G$. Then $d(v_i, v_j)\ge 3$ for any $v_i,v_j\in S_G$. 
\end{lemma}
\begin{proof}
Without loss of generality assume $i<j$. From the definition of burning process immediately follows that $d(v_i, v_j)\ge 2$.   Assume $d(v_i, v_j)=2$, i.e. there is  path $v_i, a_1, w, a_2, v_j$ and $w$ is not a burning source. Further,  $\lambda(w)\ne i$ and $\lambda(w)\ne j$ by Corollary \ref{cor:l3.4}. Moreover, $\lambda(w)\ne i-1$ by definition of burning process ($v_i \not \in U_i $). Hence , by Lemma  \ref{l3.4} and since $i<j$, we get $\lambda(v_i)=i$, $\lambda(w)=i+1$, $j=\lambda(v_j)=i+2$, which contradicts the fact that $(v_j)$ is a burning source (i.e. $v_{i+2}\not\in U_{i+2}\supset N_2(v_i)$). 
\end{proof}

\begin{theorem}\label{thm:pathHom} Any path graph $P_n$  admits a burning homomorphism.

\end{theorem} 

\begin{proof}
Let  $P=P_n$ be a path graph with  vertices $1, \dots n$ and edges $\{i,{i+1}\}$, $i=1,\dots,n-1$. 

If $n$ is not divisible by 3,
consider the burning sequence $S_{P}=(1, 4, 7,  \dots, 3k+1, \dots)$. Then the corresponding burning map is a burning homomorphism, since for any vertex $v=1, \dots n$ it maps it to:
\begin{itemize}

\item
$k+1$, if $v=3k+1$,
\item
$k+2$, if $v=3k+2$,
\item
$k+2$, if $v=3k$,
\end{itemize}
i.e. $|\lambda(i) - \lambda(i+1)| = 1$ for all $i=1,\dots,n-1$ and $\lambda$ is burning homomorphism by Corollary \ref{cor:l3.4}.

If $n$ is divisible by 3, consider the burning sequence $S_{P}=(2, 5, 8,  \dots, 3k+2, \dots)$. Then the corresponding burning map is a burning homomorphism, since for any vertex $v=1, \dots n$ it maps it to:
\begin{itemize}
\item
$k+1$, if $v=3k+2$,
\item
$k+2$, if $v=3k+1$,
\item
$k+1$, if $v=3k$,

\end{itemize}
i.e. $|\lambda(i) - \lambda(i+1)| = 1$ for all $i=1,\dots,n$ also in this case and $\lambda$ is burning homomorphism by Corollary \ref{cor:l3.4}. \end{proof}

\begin{example}\label{ex:noHom1} In this  example we show that not every tree admits a burning homomorphism and from tree admitting a burning homomorphism does not follow that its subtree admits a burning homomorphism. 

Let us consider a graph $T$  as in Figure \ref{Fig::noHom}, i.e. 
$$
V_T=\{1,2,3,4,5,6\}, E_T=\{\{1,3\}, \{2,3\}, \{3,4\},\{4,5\},\{4,6\}\}.
$$
\begin{figure}[H]
\centering
\begin{tikzpicture}


\node (1) at  (2.7, 2.4) {$\bullet$};
\node (1b) at  (2.4, 2.4) {$2$};
\node (2) at  (2.7, 4.8) {$\bullet$};
\node (2b) at  (2.4, 5) {$1$};
\node (3) at  (4, 3.9) {$\bullet$};
\node (3b) at  (4, 3.6) {$3$};
\node (4) at  (6, 3.9) {$\bullet$};
\node (4b) at  (6, 3.6) {$4$};
\node (5) at  (7.3, 2.4) {$\bullet$};
\node (5b) at  (7, 2.4) {$6$};
\node (6) at  (7.3, 4.8) {$\bullet$};
\node (6b) at  (7, 5) {$5$};

\draw (1) edge[ color=black!120, thick, -] (3);
\draw (2) edge[ color=black!120, thick, -] (3);
\draw (3) edge[ color=black!120, thick, -] (4);
\draw (4) edge[ color=black!120, thick, -] (5);
\draw (4) edge[ color=black!120, thick, -] (6);

\end{tikzpicture}
  \caption{The tree $T$ for which burning homomorphism does not exist.}
\label{Fig::noHom}
\end{figure}
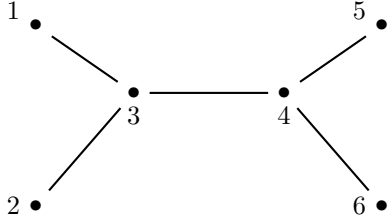
 Now we prove  that $T$ does not admit a burning homomorphism. 
If at the first moment we burn the vertex $1$, then to corresponding burning be a homomorphism we should burn vertex $5$ or $6$ at the time $t=2$, which leads to the fact that only vertex $6$ or $5$ resp. would be left for the time $t=3$, while $\lambda(4)=3$, and, hence, $\lambda$ is not a burning homomorphism.

If at the first moment we burn the vertex $3$, then at the moment $t=2$ we can burn either vertex $5$ or $6$, and the corresponding burning will be not a homomorphism, since  $\lambda(4)=2$ too.

For the burning sequences, starting from any other vertex, we will come to the same results up to isomorphism of $T$.

Note that by the definition of burning in \cite{Bonato_0} the burning sequence $(1, 3, 4)$ is allowed for $T$ and gives a burning homomorphism. This again emphasizes  the difference between our definition and definition by Bonato and coauthors.

Further, let us consider its ambient tree $Q$ (a tree, which contains $T$ as a subtree) as in Figure \ref{Fig::noHomSub}, i.e. $V_Q=V_T\cup\{7\}, E_Q=E_T\cup\{\{6,7\}\}$. 
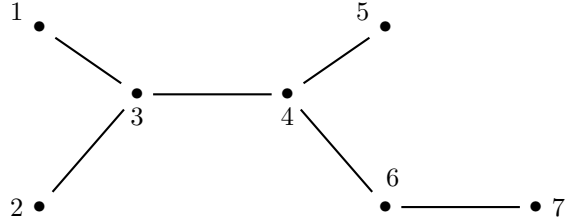
\begin{figure}[H]
\centering
\begin{tikzpicture}

\node (1) at  (2.7, 2.4) {$\bullet$};
\node (1b) at  (2.4, 2.4) {$2$};
\node (2) at  (2.7, 4.8) {$\bullet$};
\node (2b) at  (2.4, 5) {$1$};
\node (3) at  (4, 3.9) {$\bullet$};
\node (3b) at  (4, 3.6) {$3$};
\node (4) at  (6, 3.9) {$\bullet$};
\node (4b) at  (6, 3.6) {$4$};
\node (5) at  (7.3, 2.4) {$\bullet$};
\node (5b) at  (7.4, 2.8) {$6$};
\node (7) at  (9.3, 2.4) {$\bullet$};
\node (7b) at  (9.6, 2.4) {$7$};
\node (6) at  (7.3, 4.8) {$\bullet$};
\node (6b) at  (7, 5) {$5$};

\draw (1) edge[ color=black!120, thick, -] (3);
\draw (2) edge[ color=black!120, thick, -] (3);
\draw (3) edge[ color=black!120, thick, -] (4);
\draw (4) edge[ color=black!120, thick, -] (5);
\draw (4) edge[ color=black!120, thick, -] (6);
\draw (5) edge[ color=black!120, thick, -] (7);

\end{tikzpicture}
  \caption{The ambient tree $Q$ for which burning homomorphism exists.}
\label{Fig::noHomSub}
\end{figure}
Then $Q$ admits a burning homomorphism, given by a burning sequence $S=(3,7)$.
\end{example}

\begin{theorem}
For any connected graph $G=(V, E)$ with $|V|>2$ there exists a graph burning, which is not a homomorphism.
\end{theorem}

\begin{proof}
Let $v, w\in V$ are adjacent and fixed. Since graph is connected, at least one of the vertices $v, w$ has another adjacent vertex $z$. Assume w.l.o.g. that $v\sim z, z\ne w$. If $z\sim w$, then there is no a burning homomorphism by \cite[Theorem 3.10]{MMBurning2025}. By \cite[Lemma 4.5]{MMBurning2025}, there is an extension of burning $(w,z)$ to the whole graph $G$. Then for this extension $\lambda(z)=\lambda(v)=2$, i.e. it is not a homomorphism.
\end{proof}

\begin{remark}\label{r4.5} Let $T\subset Q$ be a pair of trees. It follows from \cite[Theorem 4.12]{MMBurning2025}  that any burning of $T$  can be extend 
to a burning of $Q$. For the the case of burning homomorphism on $T$, the similar statement does not hold in general. In fact,  not every tree has a burning homomorphism, but every tree has a  subtree $P_2$, which has a burning homomorphism.
\end{remark}

\section{End time of burnings}\label{S5}

In this section we describe possible end time of trees burnings in the case of burning homomorphisms. In particular, we obtain  sharp lower and upper bounds of the burning homomorphism's end time for path graphs $P_n,n\in \Bbb N$.

\begin{lemma}\label{lem::homT=n+1}
 Let $G=(V,E)$ be a  graph and $S_G$ be a burning sequence $(v_1, \dots, v_n)$ of a burning $(\lambda, S_G)$ of $G$ with the end time ${\mathbf T}$. If  $\lambda$ is a burning homomorphism, then ${\mathbf T}=n+1$.
\end{lemma}

\begin{proof}
Assume ${\mathbf T}=n$. Let $w$ be any fixed vertex, adjacent to $v_n$. Hence since $\lambda$ is a graph map to the path graph $P_{\mathbf T}, {\mathbf T}=n$ we have $\lambda(w)\in \{n, n-1\}$. If $\lambda(w)=n-1$, then $w\in  N_{n-i-1}(v_i)$ for some burning source $v_i\ne v_n$. Hence, $v_n\in N_{n-i}(v_i)\subset U_{n}$  which contradicts the Definition \ref{def::burningSeq} of a burning sequence, where $v_n\not\in U_n$. Therefore, $\lambda(w)=n=\lambda(v_n)$ and $\lambda$ is not a homomorphism by Corollary \ref{cor:l3.4}. Hence, by Theorem \ref{t3.5} we have ${\mathbf T}=n+1$.   
\end{proof}

\bigskip

 The next result Theorem \ref{thm::homPath} concerns existence and end time of burning homomorphism of disjoint union of a graph and a tree. In some sense this result can be considered as an attempt to  inverse the point of view on the burning process, since by Corollary \ref{cor:cupTi} any burning process defines a disjoint union of trees with burning homomorphisms on them.

We remind that for any graph  $G = (V, E)$  its {\em radius} $r=r(G)$ is defined as
$$
r(G)=\min_{v\in V} \max_{u\in V} d(v,u).
$$
The {\em center } of $G$ is the set of vertices $O=O(G)$ defined by
$$
O(G)=\{v\in V \;\mid\;\max_{u\in V} d(v,u)=r(G)\}.
$$

\begin{theorem}\label{thm::homPath}
Let $G = (V_G, E_G)$ be a graph and $T= (V_T , E_T)$ be a tree. Let $\lambda_G$ be a burning homomorphism of $G$ with a burning sequence $S_G=(v_1,\dots, v_n)$. If the radius of $T$ satisfies $r(T)\le n+1$, then, for any fixed vertex  $v_0\in O(T)$, the sequence $S_{G+T}=(v_0, v_1,\dots, v_n)$, defines a burning homomorphism $\lambda_{G+T}$ of the graph $G+T$.
\end{theorem}

\begin{proof}
Let $S_G$ defines the filtration $\{U^G_j\}_{j=1}^n$ of the graph $G$. Note, that since $G+T$ is a disjoint union, $n$-th closed neighborhood of any vertex $v\in V_{G+T}$  coinsides with its  $n$-th closed neighborhood in $G$, if $v\in V_G$ or in $T$, if $v\in T$. Then the filtration  $\{U^{G+T}_j\}_{j=2}^{n+2}$ of the graph $G+T$ defined by $S_{G+T}$ satisfies:
\begin{align*}
&U^{G+T}_2 = N_1(v_0)\\
&U^{G+T}_j =N_{j-1} (v_0)+U^G_{j-1}, \;\;\; j=3,\dots,n+2,
\end{align*}
where $N_j(v_0)$ is a coinciding $n-$th closed neigborhoods in $T$ or $G+T$.
Since $S_G$ is a burning sequence and $G+T$ is a disjoint sum, $v_{j+1}\not \in U^{G+T}_j$ for $0\le j\le n+1$. Moreover, since $S_G$ is a burning sequence, $U^G_{n+1}=G$ and, since $r(T)\le{n+1}, v_0\in O(T)$, $N_{n+1} (v_0)=T$. Hence, 
$$
U^{G+T}_{n+2}=N_{n+1} (v_0)+U^G_{n+1}=G+T.
$$
 Therefore, $S_{G+T}$ is a burning sequence by Definition \ref{def::burningSeq}. Further, due to the $G+T$ being a disjoint sum and by definition of $S_{G+T}$, we get $\lambda_{G+T}(v)=\lambda_{G}(v)+1$ for all $v\in V_G$, i.e. the restriction of $\lambda_{G+T}$ to $G$ is a homomorhism. Moreover, it is clear, that for all $v\in V_T$ and only for them the corresponding source is $v_0$, hence, $\lambda_{G+T}$ on $T$ is a burning homomorphism by  Corollary~\ref{cor:cupTi}. Finally, again since $G+T$ is a disjoint sum, we conclude that $\lambda_{G+T}$ is a burning homomorphism.
\end{proof}

\bigskip
Now let us look closer at the end times of burning process on path graphs. Let us start with some examples.
\begin{example}\label{ex::pathEndtime}
Let $P=P_n$ be a path graph with the set of vertices $1,\dots , n$ and set of edges $\{\{i,i+1\}\}_{i=1}^{n-1}$. Then there is a complete list of burning sequences,  defining  burning homomorphisms for $n\le 7$:
\begin{itemize}
\item
$S_P=(1)$  and $S_P=(2)$ for $n=2$,
\item
$S_P=(2)$ for $n=3$,
\item
$S_P=(1,4)$ and $S_P=(4,1)$ for $n=4$,
\item
$S_P=(1,4)$ and $S_P=(2,5)$ for $n=5$,
\item
$S_P=(2,5)$, $S_P=(4,1)$ and $S_P=(3,6)$ for $n=6$,
\item
$S_P=(3,6)$, $S_P=(5,2)$, $S_P=(1,4, 7)$ and $S_P=(7, 4, 1)$ for $n=7$.
\end{itemize}
\end{example}

\begin{theorem}\label{thm::bounds}
Let  $P=P_n, n\ge 2$ be a path graph and  $S_P$ be its  burning sequence, defining a burning homomrphism $\lambda_P$. Then  
$$
\ceil{\dfrac{\sqrt{4n-3}-1}{2}}\le |S_P|\le \ceil{\dfrac n3},
$$
and the equalities can be attained for any $n$ on both sides.
\end{theorem}
\begin{proof}
The upper  bound follows from Lemma \ref{lem::ge3} and the bound is attained due to the proof of Theorem  \ref{thm:pathHom}.

For the lower bound we firstly note, that by  \cite[Proposition 3.12]{MMBurning2025} 
$$
n\le k^2+k+1,
$$
where $k=|S_P|$. From where follows that 
$$
k\ge \dfrac{\sqrt{4n-3}-1}{2},
$$
and, hence, the lower bound. Now let us prove, that the equalities can be attained. Note
 that, since $k^2+k+1$ is strictly monotonically increasing function, for any $n$ there exists only one $k=k(n)\in \Bbb N$ such that
$$
(k(n)-1)^2+(k(n)-1)+1<n\le k^2(n)+k(n)+1.
$$
Obviously, 
$$
k(n)=\ceil{\dfrac{\sqrt{4n-3}-1}{2}}.
$$
Let now $m> 0$ be such that $n=(k(n)-1)^2+(k(n)-1)+1+m$. Due to Example \ref{ex::pathEndtime}, we can assume $k(n)\ge 3\; (n> 7)$.  Let us consider $Q=P_{(k(n)-1)^2+(k(n)-1)+1}$ as a subgraph of $P=P_n$ with the vertices $m+1,m+2,\dots n$. By the proofs of   \cite[Proposition 3.11 and Proposition 3.12]{MMBurning2025} the burning sequence $S_Q=(v_1,\dots,v_{k-1})$
\begin{align*}
v_1&=m+k(n)\\
v_j &= m+\left(\sum_{i=k(n)-(j-1)}^{k(n)-1} 2i\right) + k(n)- j \qquad \text{for } j \ge 2
\end{align*}
defines a burning homomorphims $\lambda_Q$ on $Q$.  Now we define the burning sequence $S_P$, depending on the value of $m$. Note, that $|S_P|=k(n)$ in all the cases below.  We denote further in this proof $k:=k(n)$.

\begin{itemize}
\item
Let $m\ge k$. Then $S_P=(m-k+1, v_1,\dots,v_{k-1})$ defines a burning homomorphism $\lambda_P$ on $P$. Indeed, we have 
\begin{align*}
\lambda_P&(i)=m-k+2-i, \;i\le m-k+1,\\
\lambda_P&(i)=i-(m-k),\; m-k+1<i\le m,\\
\lambda_P&(m+1)=k+1=\lambda_Q(m+1)+1,
\end{align*}
since $\lambda_Q(m+1)=2+d(v_1,m+1)=2+d(m+k,m+1)=k+1$, and, further,
\begin{align*}
\lambda_P(i)&=\lambda_Q(i+1)+1, i\ge m+1.
\end{align*}
The fact, that $\lambda_P$ is homomorphism follows, since there are just edges $\{i,{i+1}\}$, $i=1,\dots,n-1$ in the path graph and $\lambda_Q$ is a homomorphism.

Describing the corresponding burning in words: we keep the burning sequence on $Q$ (which is the right hand part of $P$), but start one moment later (this will  also shift $\lambda$ by one), while at the first moment we burn one vertex in the left hand part so, that value of $\lambda$ on the first vertex of $Q$ stays consistent. 
\item
Let $m< k$ and $m+k$ is even. Then $S_P=(1, v_1,\dots,v_{k-1})$ defines a burning homomorphism $\lambda_P$ on $P$. Indeed, we have 
\begin{align*}
\lambda_P&(i)=i, \;i\le (m+k)/2,\\
\lambda_P&( (m+k)/2+1)=(m+k)/2+1=\lambda_Q((m+1)/2+1)+1,
\end{align*}
since $\lambda_Q((m+k)/2)+1)=2+d(v_1,(m+k)/2+1)=2+d(m+k,(m+k)/2+1)=2+(m+k)/2-1$, and, further,
\begin{align*}
\lambda_P(i)&=\lambda_Q(i+1)+1, i\ge (m+1)/2+1.
\end{align*}
The fact, that $\lambda_P$ is homomorphism follows, since there are just edges $\{i,{i+1}\}$, $i=1,\dots,n-1$ in the path graph and $\lambda_Q$ is a homomorphism.

Describing the corresponding burning in words: we keep the burning sequence on $Q$ (which is the right hand part of $P$), but start one moment later (this will  also shift $\lambda$ by one), while at the first moment we burn the first vertex of $P_n$. Then we show that $\lambda_P$  between $1$ and $v_1$ is a homomorphism., due to $m+k$ being even. 

\item
Let $m< k$ and $m$ is odd. Then $S_P=(2, v_1,\dots,v_{k-1})$ defines a burning homomorphism $\lambda_P$ on $P$ for $k\ge 3$. The proof follows the same outline as in the previous case.
\end{itemize}
\end{proof}

\begin{corollary}
Let  $P=P_n$ be a path graph and $S_P$ be its burning sequence, defining a burning homomrphism $\lambda_P$ with the end time ${\mathbf T}$. Then  
$$
\ceil{\dfrac{\sqrt{4n-3}-1}{2}}+1\le {\mathbf T}\le \ceil{\dfrac n3}+1,
$$
and the equalities can be attained on both sides.
\end{corollary}

\begin{proof}
By Lemma \ref{lem::homT=n+1} we have ${\mathbf T}=|S_P|+1$.
\end{proof}

\bigskip
We finish this Section with an interesting example of the class of trees, which in some sense  burn worse, than path graphs.  More precisely, it is  known, see \cite{Bonato}, that the path graph $P_m, m={n}^2-1$ can be burned within the time $n$ with the burning sequence, satisfying $|S_P|=n-1$. Moreover, it is easy to see that there is some freedom in choosing the first vertex in the burning sequence with this time. We present an example of family of trees with $m=({n}^2-1), n\in \Bbb N$ vertices, which can be burned with the burning sequence, satisfying $|S|=n-1$ only if the burning starts from one particular vertex, while all the other burning sequences will be longer. Moreover, the corresponding burning maps are homomorphisms, while for path graphs $P_m$ there is no   burning  homomorphism with $|S_P|=n-1$, due to Theorem \ref{thm::homPath}.

\begin{example}\label{ex::T15} 
Let us define a tree $Y_m$, where $m={n}^2-1, n>1 $ with the set of vertices $1,2,\dots, m$ by firstly defining path graphs $Q_{i+1}, i=0, {n-2}$ with the set of vertices $V_{i+1}=\{ni+1, ni+n\}$ and the set of edges $E_{i+1}=\{\{j,j+1\}\}_{j=ni+1}^{ni+n-1}$, and defining path graph $Q_{n}$ as a graph with the set of vertices $V_{n}=\{n^2-n+1, n^2-1\}$ and the set of edges $E_n=\{\{j,j+1\}\}_{j=n^2+n-1}^{n^2-2}$. Now we define  the set of edges $E_{Y_m}$ as
$$
\left(\bigcup_{i=1}^{n} E_i\right)\bigcup \left(\bigcup_{i=1}^{n-1} \{1,1+ni\}\right)
$$
The corresponding graphs for $n=4\;(m=15)$ are shown  on the Figures \ref{fig::Q1Q4} and \ref{fig::T15}.
\begin{figure}[H]
\begin{tikzpicture}[scale=0.8]
\node (q1) at (1,6) {$Q_1:$};
\node (1) at (2,6) {$\bullet$};
\node (1b) at (1.7,5.7) {$1$};
\node (2) at (3.5,6) {$\bullet$};
\node (2b) at (3.5,5.7) {$2$};
\node (3) at (5,6) {$\bullet$};
\node (3b) at (5,5.7) {$3$};
\node (4) at (6.5,6) {$\bullet$};
\node (4b) at (6.5,5.7) {$4$};

\node (q2) at (1,4.5) {$Q_2:$};
\node (5) at (2,4.5) {$\bullet$};
\node (5b) at (1.7,4.2) {$5$};
\node (6) at (3.5,4.5) {$\bullet$};
\node (6b) at (3.5,4.2) {$6$};
\node (7) at (5,4.5) {$\bullet$};
\node (7b) at (5,4.2) {$7$};
\node (8) at (6.5,4.5) {$\bullet$};
\node (8b) at (6.5,4.2) {$8$};

\node (q3) at (1,3) {$Q_3:$};
\node (9) at (2,3) {$\bullet$};
\node (9b) at (1.7,2.7) {$9$};
\node (10) at (3.5,3) {$\bullet$};
\node (10b) at (3.5,2.7) {$10$};
\node (11) at (5,3) {$\bullet$};
\node (11b) at (5,2.7) {$11$};
\node (12) at (6.5,3) {$\bullet$};
\node (12b) at (6.5,2.7) {$12$};

\node (q4) at (1,1.5) {$Q_4:$};
\node (13) at (2,1.5) {$\bullet$};
\node (13b) at (1.7,1.2) {$13$};
\node (14) at (3.5,1.5) {$\bullet$};
\node (14b) at (3.5,1.2) {$14$};
\node (15) at (5,1.5) {$\bullet$};
\node (15b) at (5,1.2) {$15$};

\draw (1) edge[ color=black!120, thick, -] (2);
\draw (2) edge[ color=black!120, thick, -] (3);
\draw (3) edge[ color=black!120, thick, -] (4);

\draw (5) edge[ color=black!120, thick, -] (6);
\draw (6) edge[ color=black!120, thick, -] (7);
\draw (7) edge[ color=black!120, thick, -] (8);

\draw (9) edge[ color=black!120, thick, -] (10);
\draw (10) edge[ color=black!120, thick, -] (11);
\draw (11) edge[ color=black!120, thick, -] (12);

\draw (13) edge[ color=black!120, thick, -] (14);
\draw (14) edge[ color=black!120, thick, -] (15);
\end{tikzpicture}
\caption{Graphs $Q_1, Q_2,Q_3,Q_4$ for Example \ref{ex::T15}}
\label{fig::Q1Q4}
\end{figure}
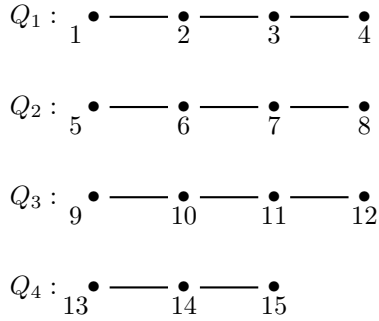

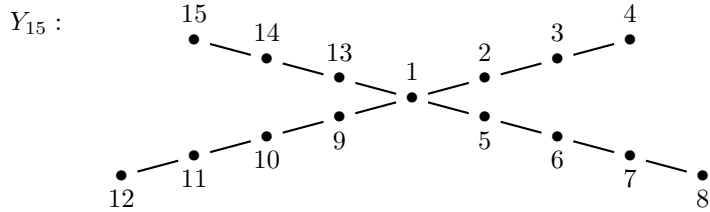
\begin{figure}[H]
\begin{tikzpicture}
\node (0) at (-5,1) {$Y_{15}:$};
\node (1) at (0,0) {$\bullet$};
\node (1b) at (0,0.35) {$1$};

\node (2) at ({cos(15)},{sin(15)}) {$\bullet$};
\node (2b) at ({cos(15)},{sin(15)+0.35}) {$2$};

\node (3) at ({2*cos(15)},{2*sin(15)}) {$\bullet$};
\node (3b) at ({2*cos(15)},{2*sin(15)+0.35}) {$3$};

\node (4) at ({3*cos(15)},{3*sin(15)}) {$\bullet$};
\node (4b) at ({3*cos(15)},{3*sin(15)+0.35}) {$4$};

\node (5) at ({cos(-15)},{sin(-15)}) {$\bullet$};
\node (5b) at ({cos(-15)},{sin(-15)-0.3}) {$5$};

\node (6) at ({2*cos(-15)},{2*sin(-15)}) {$\bullet$};
\node (6b) at ({2*cos(-15)},{2*sin(-15)-0.3}) {$6$};

\node (7) at ({3*cos(-15)},{3*sin(-15)}) {$\bullet$};
\node (7b) at ({3*cos(-15)},{3*sin(-15)-0.3}) {$7$};

\node (8) at ({4*cos(-15)},{4*sin(-15)}) {$\bullet$};
\node (8b) at ({4*cos(-15)},{4*sin(-15)-0.3}) {$8$};

\node (9) at ({cos(195)},{sin(195)}) {$\bullet$};
\node (9b) at ({cos(195)},{sin(195)-0.3}) {$9$};

\node (10) at ({2*cos(195)},{2*sin(195)}) {$\bullet$};
\node (10b) at ({2*cos(195)},{2*sin(195)-0.3}) {$10$};

\node (11) at ({3*cos(195)},{3*sin(195)}) {$\bullet$};
\node (11b) at ({3*cos(195)},{3*sin(195)-0.3}) {$11$};

\node (12) at ({4*cos(195)},{4*sin(195)}) {$\bullet$};
\node (12b) at ({4*cos(195)},{4*sin(195)-0.3}) {$12$};

\node (13) at ({cos(165)},{sin(165)}) {$\bullet$};
\node (13b) at ({cos(165)},{sin(165)+0.35}) {$13$};

\node (14) at ({2*cos(165)},{2*sin(165)}) {$\bullet$};
\node (14b) at ({2*cos(165)},{2*sin(165)+0.35}) {$14$};

\node (15) at ({3*cos(165)},{3*sin(165)}) {$\bullet$};
\node (15b) at ({3*cos(165)},{3*sin(165)+0.35}) {$15$};

\draw (1) edge[ color=black!120, thick, -] (2);
\draw (2) edge[ color=black!120, thick, -] (3);
\draw (3) edge[ color=black!120, thick, -] (4);

\draw (1) edge[ color=black!120, thick, -] (5);
\draw (5) edge[ color=black!120, thick, -] (6);
\draw (6) edge[ color=black!120, thick, -] (7);
\draw (7) edge[ color=black!120, thick, -] (8);

\draw (1) edge[ color=black!120, thick, -] (9);
\draw (9) edge[ color=black!120, thick, -] (10);
\draw (10) edge[ color=black!120, thick, -] (11);
\draw (11) edge[ color=black!120, thick, -] (12);

\draw (1) edge[ color=black!120, thick, -] (13);
\draw (13) edge[ color=black!120, thick, -] (14);
\draw (14) edge[ color=black!120, thick, -] (15);
\end{tikzpicture}
\caption{Graph  $Y_{15}$ for Example \ref{ex::T15}}
\label{fig::T15}
\end{figure}

Any  burning sequence  $S_m=(v_1, \dots,v_n)$ for $Y_m$  satisfies $v_1=1$ (otherwise there will be $n$ disconnected components left  after the first burn at any time, less than $n$, and there is no chance to have a burning sequence with $n-1$ vertices). But one can see that the sequence $S_m=(1,v_2, \dots v_{n-1})$,
where
$$
v_{i+1}=
\begin{cases}
ni+n, n+i \text{ is even }, \\
ni+n-1, n+i \text{ is odd},\\
\end{cases}
$$
gives the burning map of $Y_m$. Note, that this map is a burning homomorphism and the corresponding end time is $n$.

\end{example}

\section{Digraphs structure on burned trees}\label{S6}

Now we investigate a possibility to introduce  a natural structure of a  digraph  on a spanning tree of   a burned graph.  

We remind, that \emph{directed graph (digraph)} $G=(V_G,D_G)$ consists of a set of vertices $V_G$, and a set $D_G$ of ordered pairs $(v,w)\in D_G$ of distinct  vertices $v,w\in V_G$ called \emph{ordered edges}.  
\begin{definition}
For two digraphs $G=(V_G,D_G)$ and $H=(V_H, D_H)$ a \emph{digraph map} $f\colon G\to H$ is given by a map $f\colon V_G\to V_H$ such that 
$(v, w)\in D_G$ implies $(f(v), f(w))\in D_H$ or $f(v)=f(w)\in V_H$.  A digraph map $f$ is  a \emph{digraph homomorphism} if $(f(v), f(w))\in D_H$ for any $(v, w)\in D_G$.
\end{definition}
Let $G=(V,E)$ be a graph.  A \emph{structure of digraph on $G$} can be given 
by  equipping  with an order each edge $\{v,w\}\in E$. 

Coming back to the burning we remind, that for any graph $G=(V, E)$ any burning map $\lambda = \lambda(\mathbf B^{\mathbf T}(G))$ maps the graph onto path graph $P_{\mathbf T}$, whose vertices, if denoted by ${1, 2, \dots {\mathbf T}}$ correspond to the burning time of preimages. Hence,
in our settings a path graph 
$P = P_n, n\in \Bbb N$ has a natural structure of digraph defined by $(i, i+1)\in D_P$ for all $i = 1, \dots, n-1$.

Moreover, the following result follows almost immediately  from the proof of Theorem \ref{thm::spanning tree}.

\begin{proposition}\label{prop:digraphStruct}
 Let $T=(V_T,E_T)$ be a  tree and $(\lambda, S_T)$ be its burning. Then there is a natural structure of a directed graph on  graph $T$, induced by $\lambda$,  such that $\lambda$ is a digraph map. 
\end{proposition}

\begin{proof}
Since $T$ is its own spanning tree, $V_T=V_{\cup T_i}$, where $T_i$ are defined  as in the proof  of Theorem \ref{thm::spanning tree}. Let us take any pair of vertices  $v\sim w$ in $T$. By Lemma \ref{l3.4}, $|\lambda(v)-\lambda{w}|\in \{0,1\}$. Assume $|\lambda(v)-\lambda(w)|=1$. Then  the direction should be from the vertex with the smaller $\lambda$ to the vertex with the larger $\lambda$. Assume now that $\lambda(v)=\lambda(w)$. Then by the proof  of Theorem \ref{thm::spanning tree}  we obtain $v\in T_i, w\in T_j, j\ne i$ and we can define the directed edge as $(v,w)$, if $i<j$ and $(w,v)$ otherwise. 
\end{proof}

\begin{corollary}
Let $G$ be a graph and $(\lambda, S_G)$ be its burning. Then there is a spanning tree of $G$ with natural digraph structure, which arises from $\lambda$.
\end{corollary}

\begin{proof}
Follows from Theorem \ref{thm::spanning tree}.
\end{proof}

\begin{corollary} Let $T\subset Q$ be trees and 
 $(\lambda_T, S_T)$, $(\lambda_Q, S_Q)$ be their burnings such that the inclusion $i:T\subset Q$ is a morphism of burnings. Then $i$ is  agreed with the natural digraph structures induced by $\lambda_T$ and  $\lambda_Q$. 
\end{corollary}

\begin{proof}
 Follows from  commutative diagram \eqref{3.12}
 \end{proof}

\begin{proposition} Let $G=(V,E)$ be  an arbitrary graph and $\lambda = \lambda(\mathbf B^{\mathbf T}(G))$ be its burning homomorphism.  Then $G$ has a natural structure of  digraph such that the map $\lambda_G$ is digraph homomorphism. 
\end{proposition} 
\begin{proof}
 Follows from  the proof of Proposition \ref{prop:digraphStruct} and Corollary \ref{cor:l3.4}.
\end{proof}

\begin{example}
We note  that the given structure of digraph is not well defined for the arbitrary graph $G$ with a given burning sequence. Indeed, consider a circle $C_3$  with the vertices $1,2,3$ and all possible edges. The burning sequence $S=(1)$ does not allow to define a directed graph in the given way, since $\lambda(2)=\lambda(3)$, but $2,3\in T_1$. 
\end{example}

\section{ Strong configuration burning space }\label{S7}

A set $\mathcal B(G)$ of all burnings of any graph $G$ and a configuration space   $\Delta_G$  which is a simplicial complex associated with the set  $\mathcal B(G)$ were defined in  \cite{MMBurning2025}.  In this section   we define \emph{strong configuration burning space}  and corresponding \emph{strong burning homology groups} using a set of all  burnings homomorphisms   of a graph  and describe its relation to burning homology.   We use standard definitions of simplicial and graph constructions, see \cite{semi, MiHomotopy, Hatcher, May,  Munkres, Prasolov}. Let us firstly recall the most important concepts.  

\begin{definition} \label{def:simComp} 
A \emph{simplicial complex} $\Delta =(V,F)$ consists of a set $V$ of \emph{vertices} and a set $F$
of finite non-empty subsets of $V$ which are called \emph{simplexes}.  A simplex  $\sigma \in F$ containing $q+1 \, (q\geq 0) $ elements    is called a
$q$-\emph{simplex}  and the number $q$ is called the \emph{dimension} of $\sigma$ and we write  ($q=\operatorname{dim} \sigma$).  The set $F$ of simplexes must satisfy the following two properties:

(i) for every vertex $v\in V$  the set $\{v\}\in F$ is a simplex, 

(ii) if $\sigma \in F$ is a simplex then any   non-empty subset $\tau\subset \sigma$ is a simplex which is called  a \emph{face} of $\sigma$.

A simplex $\sigma\in \Delta$ is \emph{maximal} if it is not a face of any another simplex. 

A simplicial complex is $n$-\emph{dimensional} if it contains at least one $n$-dimen\-sional simplex but
no $(n+1)$-dimensional simplexes. 
\end{definition}

\begin{definition}\label{def:subcomplex}
A simplicial complex 
$\Delta_1=(V_1, F_1)$ is a \emph{subcomplex} of a simplicial complex  $\Delta=(V,F)$  if $V_1\subset V$ and $F_1\subset F$.

\end{definition}

A simplicial complex is called \emph{finite}, if it has  finite number of vertices.
In this paper we  consider only finite  simplicial complexes.

\bigskip

Now we come back to the burning theory. Let $G=(V,E)$ be a graph. Denote by   $\mathcal  B(G)$ the set of all burnings of $G$ and by $\mathcal  B_H(G)$ the set of all burning homomorphism of $G$. Since we consider only finite graphs, the set  $\mathcal  B(G)$
consists  of  finite number of pairs $(\lambda_G, S_G)$. Every burning sequence $S_G$ defines a set  $\widehat S_G= \{v_1, \dots, v_k\}\subset ~V$.    It follows from Definition \ref{def:simComp}  that a simplicial complex 
$\Delta=(V,F)$  is  defined  by the set $V$ of vertices and  a subset $F^m\subset F$ of maximal simplexes of $\Delta$.  We  will say that the set of maximal simplexes $F^m$ 
\emph{generates} the set $F$ of simplexes. 
 
Recall the following definition  from  \cite{MMBurning2025}. 
\begin{definition} For a given graph $G=(V,E)$, a  \emph{burning configuration space of $G$} is  a  simplicial complex $\Delta(G)=(V_{\Delta(G)}, F_{\Delta(G)})$ where 
$V_{\Delta(G)}=V$ and $F_{\Delta(G)}$ is generated by the set of all maximal simplexes 
$$
F_{\Delta(G)}^m=\{\widehat S_G\, |\, (\lambda_G, S_G)\in \mathcal B(G)\}.
$$  
\end{definition}

Now we give a new definition which is similar to  this one.

\begin{definition} For a given graph $G=(V,E)$, we  define a \emph{strong burning configuration space of $G$} as a  simplicial complex $\Gamma(G)=(V_{\Gamma(G)}, F_{\Gamma(G)})$ where 
$$
V_{\Gamma(G)} = \{v\in V\, |\, \text{ there exists $\widehat S_G\in \mathcal B_H(G)$\; such that $v\in \widehat S_G$ }\},
$$

and $F_{\Gamma(G)}$ is generated by the set of all maximal simplexes 
$$
F_{\Gamma(G)}^m=\{\widehat S_G\, |\, (\lambda_G, S_G)\ \ \text{where $\lambda_G\in \mathcal B_H(G)$ }\}.
$$  
\end{definition}
Note that the strong burning configuration space as well as set of all burning homomorphism of a graph can be empty, for example, when graph contains an odd cycle \cite[Theorem 3.10]{MMBurning2025}.
Moreover, Example \ref{ex:noHom1} shows, that they can be empty also for trees. From the other hand, the following Proposition is an immediate consequence of Theorem \ref{thm:pathHom}.

\begin{proposition}
For any path graph the corresponding strong burning configuration space is not the empty simplicial complex.
\end{proposition}

\begin{proposition}\label{p5.7} Let $G$ be a graph with a non-empty  strong burning configuration space $\Gamma(G)$. Then  we have an inclusion 
$\Gamma(G)\to \Delta(G)$ of simplicial complexes. 
\end{proposition}  
\begin{proof} Follows from definitions of configuration spaces. 
\end{proof}

\begin{example}\label{e7.7} In this example we present strong configuration spaces 
and  configuration spaces of path graphs $P_n$ for $2\leq n \leq 5$, see Table \ref{tablicaPn}.
\begin{table}[H]
\begin{tabular}{|c|c|c|c|}
\hline
n&Graph&$\Gamma$&$\Delta$\\
\hline
2&\begin{tikzpicture}
\node (1) at (0,0) {$\bullet$};
\node (1b) at (0,-0.35) {$1$};

\node (2) at (1,0) {$\bullet$};
\node (2b) at (1,-0.35) {$2$}; 
\draw (1) edge[ color=black!120, thick, -] (2);\end{tikzpicture}&\begin{tikzpicture}
\node (1) at (0,0) {$\bullet$};
\node (1b) at (0,-0.35) {$1$};

\node (2) at (1,0) {$\bullet$};
\node (2b) at (1,-0.35) {$2$};\end{tikzpicture}&
\begin{tikzpicture}
\node (1) at (0,0) {$\bullet$};
\node (1b) at (0,-0.35) {$1$};

\node (2) at (1,0) {$\bullet$};
\node (2b) at (1,-0.35) {$2$};\end{tikzpicture}\\

\hline
3&\begin{tikzpicture}
\node (1) at (0,0) {$\bullet$};
\node (1b) at (0,-0.35) {$1$};

\node (2) at (1,0) {$\bullet$};
\node (2b) at (1,-0.35) {$2$}; 
\node (3) at (2,0) {$\bullet$};
\node (3b) at (2,-0.35) {$3$};
\draw (3) edge[ color=black!120, thick, -] (2);
\draw (1) edge[ color=black!120, thick, -] (2);\end{tikzpicture}&\begin{tikzpicture}

\node (2) at (1,0) {$\bullet$};
\node (2b) at (1,-0.35) {$2$};\end{tikzpicture}&
\begin{tikzpicture}
\node (1) at (0,0) {$\bullet$};
\node (1b) at (0,-0.35) {$1$};

\node (2) at (1,0) {$\bullet$};
\node (2b) at (1,-0.35) {$3$};
\node (3) at (0.5,-1) {$\bullet$};
\node (3b) at (0.5,-1.35) {$2$};

\draw (1) edge[ color=black!120, thick, -] (2);\end{tikzpicture}\\

\hline
4&\begin{tikzpicture}
\node (1) at (0,0) {$\bullet$};
\node (1b) at (0,-0.35) {$1$};

\node (2) at (1,0) {$\bullet$};
\node (2b) at (1,-0.35) {$2$}; 
\node (3) at (2,0) {$\bullet$};
\node (3b) at (2,-0.35) {$3$};

\node (4) at (3,0) {$\bullet$};
\node (4b) at (3,-0.35) {$4$};
\draw (3) edge[ color=black!120, thick, -] (2);
\draw (3) edge[ color=black!120, thick, -] (4);
\draw (1) edge[ color=black!120, thick, -] (2);\end{tikzpicture}&\begin{tikzpicture}

\node (1) at (0,0) {$\bullet$};
\node (1b) at (0,-0.35) {$1$}; 

\node (2) at (1,0) {$\bullet$};
\node (2b) at (1,-0.35) {$4$};
\draw (1) edge[ color=black!120, thick, -] (2);\end{tikzpicture}&
\begin{tikzpicture}
\node (1) at (0,0) {$\bullet$};
\node (1b) at (0,0.35) {$1$};

\node (4) at (1,0) {$\bullet$};
\node (4b) at (1,0.35) {$4$};
\node (3) at (0,-1) {$\bullet$};
\node (3b) at (0,-1.35) {$3$};
\node (2) at (1,-1) {$\bullet$};
\node (2b) at (1,-1.35) {$2$};

\draw (1) edge[ color=black!120, thick, -] (4);
\draw (1) edge[ color=black!120, thick, -] (3);
\draw (2) edge[ color=black!120, thick, -] (4);\end{tikzpicture}\\
\hline

5&\begin{tikzpicture}
\node (1) at (0,0) {$\bullet$};
\node (1b) at (0,-0.35) {$1$};

\node (2) at (1,0) {$\bullet$};
\node (2b) at (1,-0.35) {$2$}; 
\node (3) at (2,0) {$\bullet$};
\node (3b) at (2,-0.35) {$3$};

\node (4) at (3,0) {$\bullet$};
\node (4b) at (3,-0.35) {$4$};

\node (5) at (4,0) {$\bullet$};
\node (5b) at (4,-0.35) {$5$};
\draw (3) edge[ color=black!120, thick, -] (2);
\draw (3) edge[ color=black!120, thick, -] (4);
\draw (4) edge[ color=black!120, thick, -] (5);
\draw (1) edge[ color=black!120, thick, -] (2);\end{tikzpicture}&\begin{tikzpicture}

\node (1) at (0,0) {$\bullet$};
\node (1b) at (0,-0.35) {$1$}; 

\node (4) at (1,0) {$\bullet$};
\node (4b) at (1,-0.35) {$4$};
\draw (1) edge[ color=black!120, thick, -] (4);
\node (2) at (0,-1) {$\bullet$};
\node (2b) at (0,-1.35) {$2$}; 

\node (3) at (1,-1) {$\bullet$};
\node (3b) at (1,-1.35) {$5$};
\draw (3) edge[ color=black!120, thick, -] (2);

\end{tikzpicture}&

\begin{tikzpicture}[scale=0.3]
\coordinate (A1) at (0,0);
\coordinate (A2) at (2.5,3.4);
\coordinate (A3) at (5,0);
\node (A10) at (-0.1,0) {$\bullet$};
\node (A20) at (2.5,3.44) {$\bullet$};
\node (A30) at (5.05,0) {$\bullet$};
\node (100) at  (-0.25, 0) {$1\,\,\, \, \,\,$};
\node (102) at (2.5,4.43) {$3$};
\node (101) at (5.18,0) {$\,\,\,\, \, 5$};
\definecolor{c1}{RGB}{200,200,200}
\definecolor{c2}{RGB}{200,200,200}
\definecolor{c3}{RGB}{200,200,200}

\draw (A1) -- (A2) -- (A3) -- (A1) -- cycle;

\shade [left color=c1,right color=c2] (A1) -- (A2) -- (A3) -- (A1) -- cycle;
\node (B0) at (-0.1,-3.5) {$\bullet$};
\node (B1) at (-1,-3.5) {$4$};
\node (B2) at (5.05,-3.5) {$\bullet$};
\node (B3) at (5.7,-3.5) {$2$};

\draw (A10) edge[ color=black!120, thick, ] (B0);
\draw (B0) edge[ color=black!120, thick, ] (B2);
\draw (B2) edge[ color=black!120, thick, ] (A30);
\end{tikzpicture}\\

\hline

\end{tabular}
\caption{Strong configuration spaces
and  configuration spaces of path graphs $P_n$ for $2\leq n \leq 5$}.
\label{tablicaPn}
\end{table}
\end{example}

Let $R$ be a unitary commutative ring of coefficients. Recall that the burning homology $H_*(\mathcal B(G),R)$ of a graph $G$ is the simplicial homology of  simplicial complex $\Delta(G)$ \cite{MMBurning2025}. 
Now we define strong burning homology of a graph and describe its relation to burning homology. 

\begin{definition} \rm The \emph{strong burning homology} of a graph $G$ is the simplicial homology of  simplicial complex $\Gamma(G)$. The corresponding homology groups are denoted by $H_*(\mathcal B_H(G),R)$. 
\end{definition} 

  Let $G$ be a graph with a non-empty strong burning configuartion space. It follows immediately from Proposition \ref{p5.7}  that for $n\geq 0$ there are 
 homomorphisms 
$$
H_n(\mathcal B_H(G),R)\to H_n(\mathcal B(G),R) 
$$ 
of homology groups. 

We note that the strong homology groups can be essentially non-trivial already for a digraph with a small number of vertices. 
\begin{example}

(i) A strong burning configuration space $\Gamma(Y_8)$ of the graph $Y_8$, defined in Example
\ref{ex::T15},   is shown in Figure \ref{fig::Y8}.

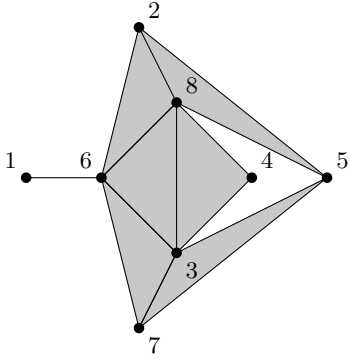
\begin{figure}[H]
\centering
\begin{tikzpicture}[scale=1,
    vertex/.style={circle,fill=black!120,inner sep=1.5pt},
    every node/.style={font=\small}
]

\coordinate (v1) at (0,0);
\coordinate (v2) at (1.5,2);
\coordinate (v3) at (2,-1);
\coordinate (v4) at (3,0);
\coordinate (v5) at (4,0);
\coordinate (v6) at (1,0);
\coordinate (v7) at (1.5,-2);
\coordinate (v8) at (2,1);
\definecolor{c1}{RGB}{200,200,200}

\fill[c1] (v2)--(v6)--(v8)--cycle;
\fill[c1] (v2)--(v5)--(v8)--cycle;
\fill[c1] (v3)--(v6)--(v8)--cycle;
\fill[c1] (v3)--(v4)--(v8)--cycle;
\fill[c1] (v3)--(v6)--(v7)--cycle;
\fill[c1] (v3)--(v7)--(v5)--cycle;

\draw (v1)--(v6);
\draw (v2)--(v6)--(v8)--(v2);
\draw (v2)--(v5)--(v8);
\draw (v3)--(v6)--(v8)--(v3);
\draw (v3)--(v4)--(v8);
\draw (v3)--(v6)--(v7)--(v3);
\draw (v3)--(v7)--(v5)--(v3);

\foreach \i in {2,4,5, 8}
{
    \fill (v\i) circle (2pt);
    \node[above right] at (v\i) {\i};
}
\fill (v1) circle (2pt);
\node[above left] at (v1) {1};

\fill (v3) circle (2pt);
\node[below right] at (v3) {3};

\fill (v6) circle (2pt);
\node[above left] at (v6) {6};

\fill (v7) circle (2pt);
\node[below right] at (v7) {7};

\end{tikzpicture}
\caption{Strong configuration space of the graph $Y_8$}
\label{fig::Y8}
\end{figure}
Hence in this case we have 
$
H_i(\mathcal B_H(Y_8),\mathbb Z)=\begin{cases} \mathbb Z,  & \text{for} \ i=0,1, \\ 
0, & \text{for  other} \ i. 
\end{cases}
$

\bigskip
(ii) The maximal simplexes of the strong burning configuration space $\Gamma(P_9)$  of the path graph $P_9$  are the following: 
$$
\{1,4,7\},  \{1,4,8\}, \{1,4,9\}, \{1,5,8\}, \{1,6,9\}, \{2,5,8\},
\{2,5,9\}, \{2,6,9\}, \{3,6,9\}.
$$
 We can write down geometrically this configuration space 
in the form 
\begin{equation}\label{7.11}
\Gamma(P_9)= \Gamma_1\cup \Delta_1\cup \Delta_2
\end{equation} 
where  $\Delta_1=\langle 1,4,9\rangle, \Delta_2=\langle 3,6,9\rangle$
are two dimensional simplexes with the corresponding vertex sets and 
the simplicial set  $\Gamma_1$   is shown in Figure \ref{fig::P9}  
\begin{figure}[H]
\centering
\begin{tikzpicture}[scale=1.4,
    vertex/.style={circle,fill=black!120,inner sep=1.5pt},
    every node/.style={font=\small}
]

\coordinate (v1) at (0,0);
\coordinate (v2) at (0,2);
\coordinate (v3) at (-2,4);
\coordinate (v4) at (1.25,0.5);
\coordinate (v5) at (0.5,1);
\coordinate (v6) at (-1,1);
\coordinate (v7) at (1,0);
\coordinate (v8) at (1,1);
\coordinate (v9) at (-0.5,1);
\definecolor{c1}{RGB}{200,200,200}

\fill[c1] (v1)--(v4)--(v7)--cycle;
\fill[c1] (v1)--(v4)--(v8)--cycle;
\fill[c1] (v1)--(v5)--(v8)--cycle;
\fill[c1] (v1)--(v6)--(v9)--cycle;
\fill[c1] (v2)--(v6)--(v9)--cycle;
\fill[c1] (v2)--(v5)--(v9)--cycle;
\fill[c1] (v2)--(v5)--(v8)--cycle;

\draw (v1)--(v4)--(v7)--(v1);
\draw (v1)--(v4)--(v8)--(v1);
\draw (v1)--(v5)--(v8);
\draw (v1)--(v6)--(v9)--(v1);
\draw (v2)--(v6)--(v9)--(v2);
\draw (v2)--(v5)--(v9);
\draw (v2)--(v5)--(v8)--(v2);

\foreach \i in {2,4, 8}
{
    \fill (v\i) circle (2pt);
    \node[above right] at (v\i) {\i};
}
\fill (v7) circle (2pt);
\node[below right] at (v7) {7};

\fill (v1) circle (2pt);
\node[below right] at (v1) {1};

\fill (v6) circle (2pt);
\node[below left] at (v6) {6};

\fill (v9) circle (2pt);
\node[above left] at (v9) {9};

\fill (v5) circle (2pt);
\node[above right] at (v5) {5};

\end{tikzpicture}
\caption{The subcomplex $\Gamma_1$ of the configuration space 
$\Gamma(P_9)$}
\label{fig::P9}
\end{figure}
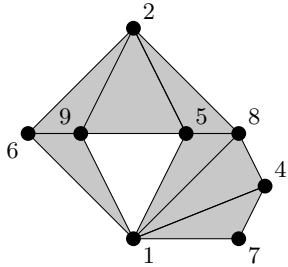
It follows immediately from (\ref{7.11}) and  Figure \ref{fig::P9}  that there is a deformation retraction of $\Gamma(P_9)$ to $\Gamma_1$. Hence 
$
H_i(\mathcal B_H(P_9),\mathbb Z)=\begin{cases} \mathbb Z,  & \text{for} \ i=0,1, \\ 
0, & \text{for  other} \ i. 
\end{cases}
$

\end{example} 

\begin{remark} It follows from Example \ref{e7.7} that  strong homology is not functorial already on the category of path graph and inclusion maps. But we  note, that  for any inclusion $\tau\colon T\to Q$ of trees  there is a  diagram 
of inclusion of configuration spaces
$$
\begin{matrix}
\Gamma(T) &\longrightarrow & \Delta(T)\\
&& \downarrow \\
\Gamma(Q) &\longrightarrow & \Delta(Q)\\
\end{matrix}
$$
which induces the corresponding diagram of homology groups and homomorphisms. 
\end{remark}

Yuri Muranov: \emph{Faculty of Mathematics and Computer Science, University
of Warmia and Mazury in Olsztyn, Sloneczna 54 Street, 10-710 Olsztyn, Poland. }

e-mail: yury.muranov@uwm.edu.pl
\smallskip

Anna Muranova: \emph{Faculty of Mathematics and Computer Science, University
of Warmia and Mazury in Olsztyn, Sloneczna 54 Street, 10-710 Olsztyn, Poland. }

e-mail: anna.muranova@uwm.edu.pl


\begin{thebibliography}{10}

\bibitem{Bonato}
A.~Bonato, \emph{A survey of graph burning}, arXiv:2009.10642v1 [math.CO]
  (2020), 1--11.

\bibitem{Burning_2014}
A.~Bonato, J.~Janssen, and E.~Roshanbin, \emph{Burning a graph as a model of
  social contagon}, Algorithms and Models for the Web Graph, 11th International
  Workshop, WAW, Proceedings (2014), 13--22.

\bibitem{Bonato_0}
A.~Bonato, J.~Janssen, and E.~Roshanbin, \emph{How to burn a graph}, Internet Mathematics \textbf{12} (2016),
  85--10.

\bibitem{Gary}
Gary Chartrand, Linda Lesniak, and Ping Zhang, \emph{Graphs and digraphs}, CRC
  Press: Boca Raton, Florida, United States, 2011.

\bibitem{semi}
Samuel Eilenberg and J.~A. Zilber, \emph{Semi-simplicial complexes and singular
  homology},  \textbf{51} (1950), 499--513.

\bibitem{MiHomotopy}
Alexander Grigor'yan, Yong Lin, Yuri Muranov, and Shing-Tung Yau,
  \emph{Homotopy theory for digraphs}, Pure and Applied Mathematics Quarterly
  \textbf{10} (2014), 619--674.

\bibitem{Mi3}
Alexander Grigor'yan, Yuri Muranov, and Shing-Tung Yau, \emph{Graphs associated
  with simplicial complexes}, Homology, Homotopy, and Applications \textbf{16}
  (2014), 295--311.

\bibitem{MiHH}
Alexander Grigor'yan, Yuri Muranov, and Shing-Tung Yau, \emph{On a cohomology of digraphs and {H}ochschild cohomology},
  Journal of Homotopy and Related Structures (2015), 209--230.

\bibitem{Hatcher}
Allen Hatcher, \emph{Algebraic {T}opology}, Cambridge University Press, 2002.

\bibitem{May}
J.~Peter May, \emph{Simplicial objects in algebraic topology}, The University
  of Chicago Press, Chicago, 1967.

\bibitem{Munkres}
R.~Munkres, James, \emph{Elements of algebraic topology}, Addison-Wesley
  Publishing Company, 1984.

\bibitem{MMBurning2025}
Yuri Muranov and Anna Muranova, \emph{Homology of graph burnings}, Topology and
  its Applications \textbf{373} (2025), 109486.

\bibitem{Prasolov}
V.~V. Prasolov, \emph{Elements of homology theory}, Graduate Studies in
  Mathematics 81. Providence, RI: American Mathematical Society (AMS), 2007.

\end{thebibliography}
\end{document}